\documentclass[12pt,reqno]{amsart}
\usepackage{graphicx} 

\usepackage[margin=2.5cm]{geometry}

\usepackage{graphicx} 
\usepackage[utf8]{inputenc}
\usepackage[english]{babel}
\usepackage[T1]{fontenc}
\usepackage{amssymb}
\usepackage{mathtools}
\usepackage{amsfonts}
\usepackage{tikz}
\usepackage{tikz-cd}
\usepackage{amsmath}
\usepackage{amsthm}
\usepackage{mathrsfs}
\usepackage{changepage}
\graphicspath{{images/}}
\usepackage{fancyhdr}
\usepackage{csquotes}
\usepackage{import}
\usepackage{xifthen}
\usepackage{pdfpages}
\usepackage{transparent}
\usepackage{tikz-cd}

\usepackage[backend=biber]{biblatex}
\usepackage{hyperref}
\hypersetup{
	colorlinks=true,
	linkcolor=blue,
	citecolor=blue,
	urlcolor=blue,
	pdftitle={}
}

\DeclareMathOperator{\diam}{diam}

\DeclareMathOperator{\Con}{Con}

\newtheorem{theorem}{Theorem}  
\newtheorem{cor}[theorem]{Corollary}  

\theoremstyle{plain}
\newtheorem{thm}{Theorem}
\newtheorem{lemma}[thm]{Lemma}

\theoremstyle{definition}
\newtheorem{definition}[thm]{Definition}

\newtheorem{remark}[thm]{Remark}
\newtheorem{exmp}[thm]{Example}

\numberwithin{equation}{section}
\numberwithin{thm}{section}

\usepackage{setspace}

\theoremstyle{remark} 
\newtheorem*{ack}{Acknowledgements}

\newcommand{\Ric}{\mathrm{Ric}}

\title{The Sunada method on Metric Measure spaces}
\author{Lewis Tadman}
\date{\today}
\address{Department of Mathematical Sciences, Durham University, United Kingdom}
\email{lewis.tadman@durham.ac.uk}
\begin{document}
\begin{abstract}
    In this paper, we provide a general Sunada--Pesce--Sutton-type method that produces pairs of isospectral metric measure spaces. This construction relies on a representation theoretic assumption, and we give examples of spaces satisfying this condition, such as RCD spaces. As an application, we construct a pair of simply connected isospectral, non-isometric RCD non-Alexandrov spaces and a pair of Alexandrov non-orbifold spaces.
\end{abstract}
\maketitle


\section{Introduction} 
In 1985, Sunada introduced a powerful method for constructing isospectral Riemannian manifolds via isometric group actions \cite{Sunada_RiemCoveringsAndIsopecMfds}. This provided a systematic solution to the problem of finding isospectral yet not isometric spaces, which were known to exist by Milnor’s work \cite{Milnor_tori_64}. More recently, Engel and Weilandt constructed the first known non-trivial isospectral, non-isometric Alexandrov spaces that are not isometric to Riemannian orbifolds via the torus method \cite{Engel_Weilandt_isospec_alex_spaces}. Sunada's method has been generalised and extended, weakening the conditions that Sunada originally assumed, and allowing the quotients to be more singular spaces (see \cite{Pesce_sunada,Sutton_10, Sutton_03}). Here, we weaken the constraints of the construction further: we allow for the spaces under consideration to be compact, geodesic, infinitesimally Hilbertian metric measure spaces, where the Laplacian on the space has compact resolvent. For such metric measure spaces, we consider the Laplacian to be the self-adjoint non-negative generator of the quadratic Cheeger energy. 

Throughout, given a metric measure space $(X,d,m)$ and a group $H$ acting on $X$ by measure-preserving isometries, we denote by $\pi_{H}\colon X \rightarrow X/H$ the quotient map, and let $m_{H}$ be the push-forward measure of $m$. We denote the Laplacian by $\Delta_{X}$. If $\lambda$ is an eigenvalue of $\Delta_{X}$, we set $ E_\lambda := \left\{ f \in D(\Delta_{X}) \ | \ \Delta_{X}f = \lambda f \ \right\}$.  For the definition of \emph{$K$-equivalence} and \emph{$G$-isospectrality}, see Definition \ref{Definition :K-equivalence} and Definition \ref{Defintion: G-isospectral} respectively.

\begin{theorem}\label{Theorem: equivarient Sunada gen}
    Let $(X_1,d_1,m_1)$ and $(X_2,d_2,m_2)$ be compact, geodesic, infinitesimally Hilbertian metric measure spaces, with Laplacians $\Delta_{X_1},\Delta_{X_2}$ respectively which have compact resolvents. Let $G$ be a compact Lie group acting by measure-preserving isometries on each $X_i$, and fix a closed subgroup $K \leq G$. Assume $X_1$ and $X_2$ are $G$-isospectral and the following condition holds:
    \begin{enumerate}
        \item[(A1)] For each $i\in \{1,2\}$, each eigenvalue $\lambda$ of $\Delta_{X_i}$, and every irreducible $G$-representation $\pi$ occurring in the eigenspace $E^i_\lambda$, we have $\dim(V_\pi^K)>0.$
    \end{enumerate}
    If $H_1,H_2\leq G$ are closed $K$-equivalent subgroups acting on $X_1$ and $X_2$ respectively, then $X_1 / H_1$ and $X_2 / H_2$ are isospectral.
\end{theorem}
If we take $K=\{e\}$, assumption (A1) is automatically fulfilled. Assumption (A1) holds in specific settings: for example, if $X_1$and $X_2$ are both Riemannian manifolds, and the actions of $G$ on these spaces have the same principal isotropy group $K \leq G$, Donnelly’s theorem  \cite[Theorem 2.3]{Donnelly} implies that any irreducible representation occurring in an eigenspace has non-zero $K$-fixed vectors; hence (A1) holds. This has been observed in \cite{Sutton_10}, with the added assumption that the submersions have minimal fibres \cite{Sutton_10}. In this setting, if we further add this assumption to Theorem~\ref{Theorem: equivarient Sunada gen}, then we recover \cite{Sutton_10} (cf.\ Remark~\ref{Remark: Comparison with Sutton}). If one takes $X_1=X_2$ and exchanges assumption (A1) for the assumption that for any irreducible representation $\pi$ of $G$, $\dim(\pi^{H_1})=\dim(\pi^{H_2})$, then one obtains a statement closer to Sunada's original theorem, but for compact Lie groups acting isometrically.
The proof of such statements splits into two parts: analytical and the representation-theoretic. The representation-theoretic part carries over to the proof of Theorem~\ref{Theorem: equivarient Sunada gen}, similar to Pesce and Sutton \cite{Pesce_sunada, Sutton_03, Sutton_10}. For the sake of completeness, we also include the representation-theoretic component of the proof. The analytical part of the proof, however, requires further work. 

Motivated by finding settings where condition (A1) is satisfied, the following corollary shows that condition (A1) is satisfied by $\mathrm{RCD}(\kappa,N)$ spaces when $K$ is the principal isotropy group. There are two main inputs for this. First, a principal orbit theorem for metric measure spaces that have ``good transport behaviour'' (\cite[Theorem 1.3]{GGKMS_qofwrcbb_18}), which generalises the principal orbit theorem for Alexandrov spaces \cite{GGG_isom_groups_of_Alexandrov_spaces_13} to the $\mathrm{RCD}$ category. The second input is the fact that given a compact $\mathrm{RCD}(\kappa,N)$ space, one can always find a Lipschitz representative of an eigenfunction \cite{Jiang_lipschitz_eigenfuctions_14}. Combining these with Theorem \ref{Theorem: equivarient Sunada gen}, one obtains the following result.

\begin{cor}\label{Corollary: Removing (A1) from Theorem A for RCD*}
    Let $(X_1,d_1,m_1)$ and $(X_2,d_2,m_2)$ be compact $\mathrm{RCD}(\kappa,N)$ spaces and let $G$ be a compact Lie group acting on each $X_i$ by measure-preserving isometries. Assume $X_1$ and $X_2$ are $G$-isospectral, and that the principal isotropy subgroups of the two actions are conjugate in $G$; fix a representative $K_0 \leq G$ of this common conjugacy class. If $H_1, H_2 \leq G$ are closed $K_0$-equivalent subgroups acting on $X_1$ and $X_2$ respectively, then $X_1/H_1$ and $X_2/H_2$ are isospectral.
\end{cor}
One of the main applications of the Sunada construction is to build isospectral spaces that are not isometric. This has been studied at length, producing examples of isospectral Riemannian manifolds and orbifolds \cite{An_Yu_Yu_13,Sunada_RiemCoveringsAndIsopecMfds,  Sutton_10}. With the above corollary in hand, one can try to produce new examples of non-isometric, isospectral $\mathrm{RCD}(\kappa,N)$ and Alexandrov spaces. The advantage of using the construction of Corollary \ref{Corollary: Removing (A1) from Theorem A for RCD*} is that one does not have to start with a smooth Riemannian manifold, and can begin with a singular space. Combined with the fact that the quotient is automatically in the same category \cite{BGP_SCBB_1_92, GGKMS_qofwrcbb_18}, one obtains the following result. 

\begin{cor}\label{Corollary :Cone statement}
Let $N>1$ and let $(X,d,m)$ be a compact $\mathrm{RCD}(N-1,N)$ space, with $\diam \leq \pi$. 
Let $G$ be a compact Lie group acting on $X$ by $m$-preserving isometries and let $K\le G$ be the principal isotropy subgroup, or $K=\{e\}$. If $H_1,H_2\le G$ are closed $K$-equivalent subgroups, then the $(1,N)$-cones with their induced cone metrics $(\Con_1^N(X/H_1),d_{\mathrm{Con}}^1)$ and $(\Con_1^N(X/H_2),d_{\mathrm{Con}}^2)$ are isospectral
compact $\mathrm{RCD}(N,N+1)$ spaces. 
\end{cor}
To give new examples using Corollary~\ref{Corollary :Cone statement}, one must ensure the isospectral pairs of $\mathrm{RCD}$ spaces produced do not belong to a more regular class, such as Riemannian manifolds, orbifolds, or Alexandrov spaces. This is achieved by choosing the bases of the cones appropriately. In this way, one may construct isospectral pairs that are RCD non-Alexandrov, and Alexandrov non-orbifold (see Examples~\ref{Example: isospectral RCD spaces} and \ref{Example: isospectral Alexandrov spaces}). However, this is not the only situation that may occur, as one may get a pair of isospectral spaces that are in two different categories: Adelstein--Sandoval \cite{Adelstein_Sandoval_17} provided an example of an orbifold and a non-orbifold which are isospectral. 
\\

Our article is organised as follows. In Section 2, we discuss the preliminaries for the paper, including the basic concepts needed from metric, $\mathrm{RCD}$, and Alexandrov geometry. Furthermore, we will define the group theoretic terms used, and state other Sunada type constructions, to put the main statements of the paper into context. Section 3 consists of the proof of Theorem~\ref{Theorem: equivarient Sunada gen}, and Section 4 comprises of the proofs of the corollaries above, and examples of non-isometric isospectral singular spaces.

\begin{ack}
    The author would like to thank his supervisor Fernando Galaz-García for his advice, suggestions and support. He is also grateful to Martin Kerin for the conversations regarding quotients with positive curvature. Finally, the author would like to thank Mohammad Alattar and Dan Disney for many helpful discussions.
\end{ack}

\section{Preliminaries}

In this section, we collect background definitions and results that will be used in the proofs of our main result.

\subsection{Metric and RCD Geometry} In this subsection, we shall review background material on $\mathrm{RCD}$ spaces we will use in the remainder of the paper. 
For a more detailed exposition on this subject, see \cite{Gigli_NSDGRicci_18,Gigli_Pasqualetto_Lecturesonnonsmoothdiffgeom_20}. 

Throughout, a metric measure space $(X,d,m)$ is defined to be a complete, separable, geodesic metric space with a non-negative Borel-measure $m$ which is finite on compact balls and has full support. We denote the space of all Lipschitz functions on $X$ by $\mathrm{LIP}(X)$.

Let $f\colon X\rightarrow\mathbb{R}$ be a Lipschitz function. We define the \textit{local Lipschitz constant} of $f$ at $x\in X$, denoted $\mathrm{Lip}f(x)$, by
$$
\mathrm{Lip}f(x):= \limsup_{y\rightarrow x}\frac{|f(x)-f(y)|}{d(x,y)}.
$$

\begin{definition}\label{Definition: Cheeger Energy}
Let $(X,d, m)$ be a metric measure space. Given $f\in L^{2}( m)$, we define its \textit{Cheeger energy} by
$$
\mathrm{Ch}(f)
:=\inf\left\{
\liminf_{n\to\infty}\frac12\int_X \bigl(\mathrm{Lip}\,f_n\bigr)^2\,d m
\;\middle|\;
f_n\in \mathrm{LIP}(X),\ f_n\to f \text{ in } L^{2}( m)
\right\}.
$$
\end{definition}

\begin{definition}\label{Definition: Sob Space}
Let $(X,d, m)$ be a metric measure space. The Sobolev space $W^{1,2}(X)$ is 
\[
W^{1,2}(X):=\{f\in L^{2}( m)\mid \mathrm{Ch}(f)<\infty\}
\]
equipped with the norm
\[
\|f\|^{2}_{W^{1,2}}:=\|f\|^{2}_{L^{2}( m)}+2\,\mathrm{Ch}(f).
\]
If this norm comes from an inner product, we say that $(X,d, m)$ is
\textit{infinitesimally Hilbertian}.
\end{definition}
We denote by $|\nabla f|$ the weak upper gradient of $f\in W^{1,2}(X)$, which satisfies $\int_X|\nabla f|^2 \ dm =2\mathrm{Ch}(f)$. Before defining the Laplacian, and $\mathrm{RCD}$ spaces, we define the following pointwise inner product.

\begin{definition}[Pointwise inner product]
Let $(X,d,m)$ be an infinitesimally Hilbertian metric measure space. Given $f,g\in W^{1,2}(X)$ we define the pointwise inner product by
$$
\langle \nabla f,\nabla g\rangle
:=\frac14\left(|\nabla(f+g)|^{2}-|\nabla(f-g)|^{2}\right).
$$
\end{definition}
With the above definitions, we can now define the Laplacian. 

\begin{definition}\label{Definition: Laplacian}
Let $(X,d, m)$ be an infinitesimally Hilbertian metric measure space. A function $f\in W^{1,2}(X)$ is said to lie
in the domain of the Laplacian, denoted $D(\Delta)$, if there exists $g\in L^{2}(m)$
such that
\[
\int_X h\, g\, d m
=
\int_X \langle \nabla f,\nabla h\rangle\, d m
\]
for all $h\in W^{1,2}(X)$.
We denote this function $g$ by $\Delta f$, and refer to it as the \emph{Laplacian} of $f$. Note that the function $g$ is unique $m$-a.e. 
\end{definition}
We now give the definition of an $\mathrm{RCD}(\kappa,N)$ space. The curvature-dimension condition was introduced in \cite{Lott_Villani_09,sturm_I_2006,sturm_II_2006}, and this hypothesis was refined further in \cite{Ambrosio_gigli_savare_RCD_orig_14}. The following definition is equivalent to that of \cite{Ambrosio_gigli_savare_RCD_orig_14}; see \cite{Erbar_Kuwada_Sturm_15} for the case $N<\infty$ and for $N=\infty$ \cite{Ambrosio_Gigli_Savare_15}. Originally, there were two definitions, $\mathrm{RCD}^*(\kappa,N)$ and $\mathrm{RCD}(\kappa,N)$; however, these were shown to be equivalent \cite{Cavalletti+Milman_21,Li_24}.

\begin{definition}
    For $\kappa\in\mathbb{R}$ and $N\in[1,\infty)$, a metric measure space $(X,d, m)$ is said to be
an $\mathrm{\mathrm{RCD}}(\kappa,N)$ \emph{space} if the following are satisfied:
\begin{enumerate}
\item $(X,d, m)$ is infinitesimally Hilbertian.
\item \emph{Volume growth property}: there exist $x\in X$ and $C>1$ such that
$$
 m\left(B_r(x)\right)\le C e^{C r^{2}},
 \ \ \  \text{for all } r>0.
$$
\item \emph{Sobolev-to-Lipschitz property}: for every
$f\in W^{1,2}(X)$ with $|\nabla f|\le 1$ $ m$-a.e.\ on $X$, there exists a
$1$-Lipschitz representative of $f$.
\item \emph{Bakry--\'Emery inequality}: for all
$f\in D(\Delta)$ with $\Delta f\in W^{1,2}(X)$ and for all $g\in D(\Delta)$ that are
non-negative and bounded with $\Delta g\in L^\infty( m)$, we have
$$
-\frac12\int_X (\Delta g)\,|\nabla f|^2\,d m
+\int_X g\,\langle \nabla(\Delta f),\nabla f\rangle\,d m
\ge
\kappa\int_X g\,|\nabla f|^2\,d m
+\frac1N\int_X g\,(\Delta f)^2\,d m.
$$
\end{enumerate}
\end{definition}
\begin{remark}
    The class defined above contains many well-studied spaces. The results of Petrunin \cite{Petrunin_11} and Zhang--Zhu \cite{Zhang_Zhu_10} show that finite-dimensional Alexandrov spaces equipped with the Hausdorff measure are $\mathrm{RCD}$. Lytchak and Stadler have shown that the class of non-collapsed RCD and Alexandrov spaces coincide in dimension $2$ \cite{Lytchak_Stadler_23}. The class of $\mathrm{RCD}$ spaces is strictly larger than that of Alexandrov spaces. For example, Ricci limit spaces belong to this class, yet not all such spaces are Alexandrov. 
\end{remark}

We will also need the notion of quotients of metric measure spaces via a compact Lie group $G$ acting by measure-preserving isometries. For more on metric measure spaces under compact Lie group actions, see \cite{GGKMS_qofwrcbb_18}. Let $(X,d,m)$ be a geodesic metric measure space, and let $\pi_G\colon X \rightarrow X/G$ be the quotient map, where $G$ is a compact Lie group acting on $X$ by measure-preserving isometries. Given $x\in X$, we let $[x]:=\pi_G(x)$. We equip the space $X/G$ with the push-forward measure $m_G:= m\circ(\pi_G)^{-1}$ and the quotient metric
$$
d_G([x],[y]):=\inf_{\substack{u\in \pi_G^{-1}([x])\\ v\in \pi_G^{-1}([y])}} d(u,v).
$$

It is known that the group of measure-preserving isometries of an $\mathrm{RCD}$ space is a Lie group \cite{Guijarro_santos_RCD_iso_group_19,Sosa_LiegroupofRCD_18} which is compact if the $\mathrm{RCD}$ space is compact. We will rely on the following fact throughout the paper: if the base space is an $\mathrm{RCD}(\kappa,N)$ space, then the quotient space remains in that class.

\begin{thm}[\protect{\cite[Theorem 1.1]{GGKMS_qofwrcbb_18}}] Let $(X,d,m)$ be an $\mathrm{RCD}(\kappa,N)$ space. If $G$ is a compact Lie group acting by measure-preserving isometries on $X$, then the quotient metric-measure space $(X/G,d_G,m_G)$ is an $\mathrm{RCD}(\kappa,N)$ space. 
\end{thm}

Let $(X,d,m)$ be a metric measure space, and let $G$ act on $X$ via isometries. The \textit{isotropy group} of $x\in X$ is the subgroup $G_x:=\{g\in G \ | gx=x\} \leq G$. The principal orbit theorem for RCD spaces follows from \cite{GGKMS_qofwrcbb_18}, where this result was proven in the more general case of metric measure spaces that satisfy the \emph{good transport behaviour} (GTB) condition; see \cite{GGKMS_qofwrcbb_18,Kell2017} for more details on this assumption. $\mathrm{RCD}$ spaces are known to satisfy the (GTB) assumption \cite{Cavalletti_Huesmann_15, Cavalletti_Mondino_17, GGKMS_qofwrcbb_18, Gigli_Rajala_Sturm_16, Kell2017}.  

\begin{thm}[\protect{\cite[Theorem 1.3]{GGKMS_qofwrcbb_18}}] Assume that $(X,d,m)$ has $(\mathrm{GTB})$. Then there exists (up to conjugation)
a unique subgroup $G_{\min}\le G$ such that $G_x$ is conjugate to $G_{\min}$ for $m$-almost every $x\in X$.
\end{thm}
We call $G_{\min}$ the \emph{principal isotropy} of the action of $G$ on $(X,d, m)$. An orbit $G/G_x$ is said to be \emph{principal} if the isotropy of $x$ is (conjugate to) $G_{\min}$.

 We now define the \emph{spherical $(1,N)$-cone} over a metric measure space; see \cite[Definition 5.1]{Ketterer_Conesovermetricmeasurespaces_15} for $(K,N)$-cones. 

\begin{definition}\label{Definition: Spherical $N$-cone (spherical suspension)}
Let $(X,d,m)$ be a metric measure space with $\diam(X)\le \pi$ and let $N\ge 1$.
The \textit{spherical $N$-cone} (equivalently, the \textit{spherical suspension}) over $(X,d,m)$
is the metric measure space
\[
\Con_1^{N}(X):=\bigl( (X\times[0,\pi])/\!\sim,\ d_{\Con},\ m_{\Con}\bigr),
\]
where $(x,0)\sim(y,0)$ for all $x,y\in X$ and $(x,\pi)\sim(y,\pi)$ for all $x,y\in X$.
We write points as $[x,t]$.
For $[x,t],[y,s]\in \Con_1^{N}(X)$ set
$$
d_{\Con}\left([x,t],[y,s]\right)
:=\arccos\left(\cos t\,\cos s+\sin t\,\sin s\cos\left(d(x,y)\right)\right).
$$
Let $\pi: X\times[0,\pi]\to \Con_1^{N}(X)$ be the quotient map and define the suspension measure to be
$$
m_{\Con}:=\pi_{\#}\left(m\otimes \sin^{N}(t)\,dt\right).
$$
\end{definition}

We conclude this subsection with a special case of a theorem of Ketterer (setting $K=1)$, which asserts that the spherical suspension of an $\mathrm{RCD}$ space remains an $\mathrm{RCD}$ space.

\begin{thm}[\protect{\cite{Ketterer_Conesovermetricmeasurespaces_15}}]
Let $N\ge 1$ and let $(X,d,m)$ satisfy $\mathrm{RCD}(N-1,N)$ and $\diam(X)\le \pi$.
Then the spherical $N$-cone $\Con_1^{N}(X)$ is $\mathrm{RCD}(N,N+1)$.
\end{thm}

\subsection{Alexandrov Geometry}\label{Subsection: Alexandrov Geometry} In this subsection, we give the basic definitions and properties which we will use regarding Alexandrov spaces. For more on these spaces, we refer the reader to \cite{BBI_course_on_metric_geometry_01, BGP_SCBB_1_92}. 

\begin{definition}[Alexandrov space]\label{Definition: Alexandrov space}
    Let $(X,d)$ be a complete, locally compact length space. For $x,y\in X$, let $xy$ denote a minimal geodesic from $x$ to $y$, which is parametrised proportionately by arc length. Let $\triangle_{xyz}$ denote a triangle with vertices $x,y,z\in X$, where the sides are the minimal geodesics $xy, yz, zx$. We say that $(X,d)$ is an \textit{Alexandrov space} of curvature bounded below by some $\kappa\in \mathbb{R}$, if 
    for any admissible triangle $\triangle_{x_1x_2x_3}$ in $X$, there exists a triangle $\tilde{\triangle}_{x_1x_2x_3}$ in a simply connected space form $(S_\kappa,d_\kappa)$ of constant curvature $\kappa$, such that $d(x_i,x_j)=d(\tilde{x_i},\tilde{x_j})$, for all $i,j$, such that 
    $$
d(x_1x_2(s),x_1x_3(t)) \geq d_\kappa(\tilde{x_1}\tilde{x_2}(s),\tilde{x_1}\tilde{x_3}(t)),
    $$
    for any $s,t\in[0,1]$.
\end{definition}
\begin{definition}[Space of Directions] \label{Definition: Space of directions}
    Let $X$ be a length space, and assume that angles exist between shortest paths. The set of geodesics emanating from the point $p$, where we identify the geodesics whose angle between them is zero, is denoted
    $$
\Sigma_p' := \{ [pq] : q \in X \setminus \{p\} \} /\sim,
    $$
    where for two geodesics $\gamma,\gamma'\colon[0,1] \rightarrow X$, with $\gamma(0)=\gamma'(0)=p$ and $\gamma(1) = x, \gamma'(1) =y$, then $\gamma \sim \gamma'$ if $\measuredangle xpy =0 $.
    Additionally, since we know that the angles exist between all shortest paths, the angle between them defines a metric on $\Sigma_p'$. However, this metric space may not be complete, hence let $\Sigma_p$ denote the completion of $(\Sigma_p',\angle )$ in the sense of a metric completion. We define $(\Sigma_p,\angle)$ to be \textit{the space of directions at }$p$.
\end{definition}
The space of directions is a fundamental concept in Alexandrov geometry. This notion can be used to obstruct a space from being Alexandrov. That is, given any Alexandrov space, the space of directions must be an Alexandrov space with curvature bounded below by $1$ (which we abbreviate to $\mathrm{CBB}(1))$; see \cite{BGP_SCBB_1_92}. We will also rely on the fact that submetries preserve lower curvature bounds; this includes quotient maps \cite{BGP_SCBB_1_92}.

\subsection{Group Theory} This subsection is dedicated to the group theory that we will rely on in the remainder of the paper. We recall first that, in \cite{Sutton_10}, Sutton defines \emph{equivariant isospectrality} for Riemannian manifolds. Here, we modify this definition to include the case where the space is a compact infinitesimally Hilbertian metric measure space with compact resolvent. We will then define the notion of \emph{$K$-equivalence}.

Let $(X,d, m)$ be an infinitesimally Hilbertian metric measure space with compact resolvent, and let $\Delta_X$ be the Laplacian of $(X,d,m)$ (as defined in Definition \ref{Definition: Laplacian}). Let $G$ be a compact Lie group acting on $X$ by $m$-preserving isometries, and let $\tau^{G}$ be the induced unitary representation on $L^{2}(X, m)$ given by
$$
(\tau^{G}(g)f)(x) = f\!\left(g^{-1}\cdot x\right),
$$
where $g\in G,\ f\in L^{2}(X, m)$. Let $\widehat G$ denote the set of equivalence classes of irreducible unitary representations of $G$. For each $[\rho]\in \widehat G$, define $L^{2}_{\rho}(X)\subset L^{2}(X,m)$ to be the closed linear span of all irreducible $G$-invariant subspaces $W\subset L^{2}(X,m)$ on
which $\tau^{G}$ is equivalent to $\rho$. Then
$$
L^{2}(X, m) = \bigoplus_{[\rho] \in \widehat G} L^{2}_{\rho}(X),
$$
an orthogonal decomposition into $G$-invariant closed subspaces. 

\begin{definition} \label{Defintion: G-isospectral}
Let $(X_1,d_1, m_1)$ and $(X_2,d_2, m_2)$ be two metric measure spaces with $m$-preserving isometric $G$-actions, with Laplacians $\Delta_{X_1},\Delta_{X_2}$ with compact resolvent, each commuting with the corresponding
induced $G$-action. We say that $X_1$ and $X_2$ are \textit{$G$-isospectral} if, for every
$[\rho]\in \widehat G$,
$$
\mathrm{Spec}\!\left(\Delta_{X_1}\big|_{L^{2}_{\rho}(X_1)}\right)
=
\mathrm{Spec}\!\left(\Delta_{X_2}\big|_{L^{2}_{\rho}(X_2)}\right),
$$
with multiplicity.
\end{definition}

\begin{remark}\label{Remark: equivillent G-isospec def} One can also state the following equivalent intertwining formulation of the definition above, which was noted in \cite{Sutton_10}. The spaces above are $G$-isospectral if and only if there exists a unitary operator
$U:L^{2}(X_{1},m_{1})\to L^{2}(X_{2},m_{2})$ such that
$$
U \Delta_{X_1} = \Delta_{X_2} U,
 \ \text{and} \ \ 
U\,\tau^{G}_{1}(g) = \tau^{G}_{2}(g)U\, \ \ \  \forall g\in G.
$$
\end{remark}

\begin{definition}[$K$-equivalence]\label{Definition :K-equivalence}
Let $G$ be a compact Lie group and let $K\leq G$ be a closed subgroup. 
Two closed subgroups $H_1,H_2\leq G$ are called \textit{$K$-equivalent} if
$$
\dim\left(V_\tau^{H_1}\right)=\dim\left(V_\tau^{H_2}\right)
 \ \ \text{for every irreducible }\tau\in\widehat{G}\text{ with }\tau^{K}\neq 0.
$$
\end{definition}

\subsection{Sunada-type Constructions} In this final subsection, we recall the relevant work concerning the Sunada method, to put the theorems in this paper in context. We begin with Sunada's original theorem, inspired by Galois theory and Galois towers. One can see this via diagram \eqref{diag:sunada} below the following theorem. Subsequently, we state Sutton's generalisation of Sunada's construction. This method of engineering isospectral non-isometric spaces has been very effective, see \cite{Adelstein_Sandoval_17, An_Yu_Yu_13, Gordon_Survey_isospec_00, Sunada_RiemCoveringsAndIsopecMfds}. Other extensions, applications and discussions of the Sunada-type constructions are also available, see \cite{Adelstein_Sandoval_17, Gordon_2_dec_on_09, Qin_spectra_19,Smit_Gornet_sutton-covering-spec-10,Sutton-specgeom-symetry-21}.

\begin{thm}[Sunada \cite{Sunada_RiemCoveringsAndIsopecMfds}, 1985] \label{Thm: Sunada '85}
Let $M$ be a compact Riemannian manifold, and let $\pi\colon  M \rightarrow M/G$ be a normal finite Riemannian covering, with covering transformation group $G$. Let $M_i:=M/H_i$, and $\pi_1 \colon M_1 \rightarrow M/G$ and $\pi_2\colon M_2 \rightarrow M/G$ be the coverings of the corresponding subgroups $H_1$ and $H_2$ of $G$ respectively. If the conjugacy classes of $G$ intersect $H_1$ and $H_2$ the same number of times, then the zeta functions of $M_1$ and $M_2$ are equal, i.e.\ the spaces $M_1$ and $M_2$ are isospectral. 
\end{thm}

\begin{equation}
\label{diag:sunada}
\begin{tikzcd}[row sep=large, column sep=huge]
& M \arrow[dl, two heads, swap, "/H_1"] 
    \arrow[dr, two heads, "/H_2"] 
    \arrow[dd, two heads, "/G"] & \\
M_1 \arrow[dr, two heads, swap, "\pi_1"] & & 
M_2 \arrow[dl, two heads, "\pi_2"] \\
& M/G &
\end{tikzcd}    
\end{equation}

Pesce generalised Sunada’s construction by formulating the method in terms of Lie group actions and relative equivalence of representations \cite{Pesce_sunada}. In particular, for a Riemannian manifold equipped with an isometric action of a Lie group $G$, Pesce replaced Sunada’s almost-conjugacy condition by the weaker condition of $K$-equivalence, where $K$ is the generic stabiliser of the action. Sutton \cite{Sutton_03} extended the Sunada--Pesce method further by allowing quotients by non-trivial closed, and in particular connected, subgroups of a compact Lie group. Later, Sutton \cite{Sutton_10} developed an equivariant version of the construction, allowing one to start from two possibly different Riemannian $G$-manifolds, provided that they are $G$-isospectral and have the same generic stabiliser. For more on how representation theory and spectral geometry interact, we refer the reader to \cite{Sutton-thesis}.

\begin{thm}[Sutton \cite{Sutton_10}, 2010]\label{Thm: Sutton '10} Let $M_1$ and $M_2$ be two compact isospectral Riemannian manifolds and let $G$ be a compact Lie group such that:
\begin{enumerate}
    \item $G$ acts by isometries on $M_1$ and $M_2$,
    \item $M_1$ and $M_2$ are equivariantly isospectral with respect to $G$,
    \item The actions of $G$ on $M_1$ and $M_2$ have the same generic stabiliser $K \leq G$.
    \end{enumerate}
    Furthermore, let $H_1$ and $H_2$ be two closed subgroups of $G$ that are $K$-equivalent and act on $M_1$ and $M_2$ respectively such that their Riemannian submersions to their quotients, with the natural induced metric, have minimal fibres. Then $M_1/H_1$ and $M_2/H_2$ are isospectral on functions.
\end{thm}

\section{Proof of Theorem~\ref{Theorem: equivarient Sunada gen}}
The proof strategy of Theorem \ref{Theorem: equivarient Sunada gen} is similar to that in \cite{Parzanchevski_13, Pesce_sunada, Sutton_10, Sutton_03}. Fix $i \in \{1,2\}$ and let $H \leq G$ be closed. We know that the action preserves measures via isometries, and since $X_i$ is an infinitesimally Hilbertian metric measure space, by \cite[Section 5]{GGKMS_qofwrcbb_18} the pull-back by $\pi_{i,H}$ gives a unitary identification $L^2(X_i/H,m_{i,H}) \cong L^2(X_i,m_i)^H$, under which the Cheeger energy of $X_i/H$ agrees with the restriction of the Cheeger energy of $X_i$ to $H$-invariant functions. Hence, by uniqueness of the non-negative self-adjoint operator associated to a closed quadratic form, the quotient Laplacian is unitarily equivalent to the restriction of $\Delta_{X_i}$ to $L^2(X_i,m_i)^H$. This operator thus inherits a discrete spectrum. In particular, if $\lambda$ is an eigenvalue of the quotient Laplacian, then its multiplicity satisfies
$\mathrm{mult}_\lambda(X_i/H)=\dim(E_\lambda^i)^H$, where 
$$
E_\lambda^i := \left\{ f \in D(\Delta_{X_i}) \ | \ \Delta_{X_i}f = \lambda f \ \right\}.
$$

Now, since $(X_1,d_1,m_1)$ and $(X_2,d_2,m_2)$ are $G$-isospectral, we have a unitary map 
$$
U\colon L^2(X_1,m_1) \rightarrow L^2(X_2,m_2),
$$
that intertwines the Laplacians and the $G$-representations, see Remark \ref{Remark: equivillent G-isospec def}. Furthermore, the map 
$$
U\biggr |_{E^1_\lambda}\colon E^1_\lambda\rightarrow E^2_\lambda,
$$
is a representation isomorphism, so the multiplicity of each irreducible representation $\pi$ in $E^1_\lambda$ is equal to the multiplicity of that irreducible representation in $E^2_\lambda$, that is, $m_\lambda^1(\pi)=m_\lambda^2(\pi)$.

By assumption, each Laplacian has compact resolvent (hence discrete eigenvalues), and $G$ is compact, so we can decompose $E^i_\lambda$ as 
$$
E^i_\lambda \cong \bigoplus_{\pi\in \hat{G}} m^i_\lambda(\pi) \  V_\pi,
$$
where $\hat{G}$ denotes the set of irreducible representations of $G$, each $m^i_\lambda(\pi)$ is finite, and only finitely many terms are non-zero. Taking $H$ fixed vectors gives
$$
\dim(E_\lambda^i)^H = \sum_{\pi\in \hat{G}} m^i_\lambda(\pi) \dim(V_\pi^H).
    $$
Moreover, note that the coefficients $m^i_\lambda(\pi)$ are the same for both $i=1,2$, as remarked in the preceding paragraph.  

Define $\hat{G}_K :=\{\pi \in \hat{G} \ | \ V_\pi^K \neq 0 \}$. By assumption (A1), if $m^i_\lambda(\pi)\neq0$, then $\pi \in\hat{G}_K$, and thus
$$
\dim(E^i_\lambda)^H = \sum_{\pi \in \hat{G}_K} m^i_\lambda(\pi) \dim(V_\pi^H).
$$
Since $H_1,H_2$ have the property that for every $\pi$, with $V_\pi^K\neq0$, $\dim(V_\pi^{H_1}) = \dim(V_\pi^{H_2})$, we can see that $\dim(E^1_\lambda)^{H_1} = \dim(E^2_\lambda)^{H_2}$. Therefore, we can conclude that $(X_1/H_1,d_{1,H_1},m_{1,H_1})$ is isospectral to $(X_2/H_2,d_{2,H_2},m_{2,H_2})$. \qed

\begin{remark}\label{Remark: Comparison with Sutton}
    The main difference between Theorem~\ref{Theorem: equivarient Sunada gen} and either of Sutton's generalisations of the Sunada technique \cite{Sutton_10, Sutton_03}, when the base space is a smooth manifold,  is the measure that one equips the quotient space with. In Theorem \ref{Theorem: equivarient Sunada gen}, we give the quotient the push-forward measure, whereas in Sutton’s setting the quotient is equipped with the natural volume measure on the quotient space. With the push-forward measure, Theorem \ref{Theorem: equivarient Sunada gen} gives that the Laplacians in the quotient will be isospectral, even with a possible drift term from the intrinsic Laplacian on the quotient space, coming from the volume form of the quotient. For Sutton to relate the intrinsic Laplacian on the quotient space to the Laplacian on the base space, the author requires minimal fibres to conclude that the quotient spaces are isospectral. If one asks for minimal fibres in the setting of a smooth base manifold, in Theorem \ref{Theorem: equivarient Sunada gen}, then the drift term vanishes, and one recovers the intrinsic Laplacian of the quotient, associated with the natural volume form on the quotient space, recovering Sutton's theorem. Thus, the measures which the quotients are equipped with are not necessarily the `canonical' measures of the space, but are the weighted measures arising from the push-forward measure of the base space.
\end{remark}
\section{Proofs of Corollaries and Applications}
\begin{proof}[Proof of Corollary \ref{Corollary: Removing (A1) from Theorem A for RCD*}] We know that the Laplacian on compact $\mathrm{RCD}$ spaces has compact resolvent \cite{Ambrosio_Honda_Portegies_21}. These spaces are geodesic and by definition infinitesimally Hilbertian \cite{Ambrosio_gigli_savare_RCD_orig_14}. Furthermore, using \cite[Section 5]{GGKMS_qofwrcbb_18}, we have for $i=1,2$, for any closed subgroup $H\leq G$, $L^2(X_i/H, m_{i,H}) \cong L^2(X_i,m_i)^H$. Thanks to \cite[Theorem 5.12]{GGKMS_qofwrcbb_18}, one obtains that the Laplacian of the quotient is the restriction of the Laplacian of the base space, restricted to $H$-invariant functions. $\mathrm{RCD}$ spaces have the (GTB) property, thus \cite[Theorem 1.3]{GGKMS_qofwrcbb_18} tells us that there exists $G_{\mathrm{min}}\leq G$, unique up to conjugacy, such that the orbit is of type $G/G_{\mathrm{min}}$, for $m_i$-almost everywhere. We take $K_0 = G_{\mathrm{min}}$ and define the principal set to be 
$$
\mathcal{P}_i=\{ x\in X_i \ | G_x \text{ is conjugate to } K_0 \},
$$
noting that $m_i(X_i\setminus \mathcal{P}_i) =0$.

Fix $i\in\{ 1,2\}$, $\lambda$ an eigenvalue of $\Delta_{X_i}$, and an irreducible $G$-representation $\pi$ occurring in the eigenspace $E_\lambda^i$. Let us show that $\dim \pi^{K_0} >0$. Since $E^i_\lambda$ is a finite-dimensional $G$-representation, we know that it decomposes into a finite direct sum of irreducible representations. Thus there is a $G$-invariant subspace $W\subset E^i_\lambda$ such that $(W,\rho|_W) \cong (V_\pi,\pi)$. Pick a non-zero function $u \in W$. Using \cite{Jiang_lipschitz_eigenfuctions_14}, there exists a Lipschitz representative of $u$ , and thus the non-zero set of the representative of $u$ is open. Since $\mathrm{supp}(m_i) = X_i$, any non-empty open set has positive measure. Thus, the non-zero set intersects the principal set with full measure. Choose a principal point $x$, with $u(x)\neq0$. Then we know that $x$ has principal isotropy group conjugate to $K_0$. Now, we will average $u$ over the compact group $G_x$:
$$
u_{G_x} : = \int_{G_x} (g\cdot u) \ dg 
$$
where $dg$ is the normalised Haar measure. This quantity is $G_x$-invariant by construction. Furthermore, $u_{G_x}(x)=u(x)\neq0$, since $g\cdot u$ evaluated at $x$ is $(g \cdot u)(x) = u(g^{-1}x)=u(x)$. Thus the $\pi$-space contains a non-zero $G_x$-fixed vector, and thus $\pi^{G_x}\neq0$. $G_x$ is conjugate to $K_0$, thus $\pi^{K_0}\neq0$, proving that $\dim \pi^{K_0} >0$, as we wanted.

Thus, for compact $\mathrm{RCD}(\kappa,N)$ spaces, if the subgroup $K$ in Theorem \ref{Theorem: equivarient Sunada gen} is chosen to be the principal isotropy subgroup, then (A1) is automatically satisfied. Applying Theorem \ref{Theorem: equivarient Sunada gen} yields the Corollary.
\end{proof}
\begin{proof}[Proof of Corollary \ref{Corollary :Cone statement}]
    It is known that the $(1,N)$-cone over an $\mathrm{RCD}(N-1,N)$ space is $\mathrm{RCD}(N,N+1)$ \cite{Ketterer_Conesovermetricmeasurespaces_15}. Furthermore, we can act isometrically on the suspension, by acting only on the $X$ factor, that is, for $g\in G$ and $[x,t] \in \Con_1^N(X)$, then the induced action on the spherical cone is $g \cdot [x,t]= [g \cdot x,t]$. Thus, we identify $\Con_1^N(X/H_i) \cong \Con_1^N(X)/H_i$ isometrically as metric measure spaces. Via \cite[Theorem 1.1]{GGKMS_qofwrcbb_18}, the quotients are also $\mathrm{RCD}(N,N+1)$. If $K=\{ e\}$, then Theorem \ref{Theorem: equivarient Sunada gen} gives the corollary. If $K$ is a representative of the principal isotropy group, applying Corollary \ref{Corollary: Removing (A1) from Theorem A for RCD*} to $\Con_1^N(X)$, we know that the quotients $\Con_1^N(X/H_1)$ and $\Con_1^N(X/H_2)$ are isospectral. 
\end{proof}

\begin{remark}\label{Remark: Cone points remark}
     Using Corollary \ref{Corollary :Cone statement}, one may tailor oneself into constructing RCD non-Alexandrov, or Alexandrov non-orbifold spaces.  In the following, by \emph{orbifold} we mean Riemannian orbifold, and CBB$(k)$ denotes curvature bounded below by $k$ in the Alexandrov triangle comparison sense. Furthermore, by \emph{spherical orbifold}, we mean a Riemannian orbifold of constant sectional curvature $1$.
Assuming that $X/H_i$ is not a CBB(1) space, then the image of the spherical suspension point of $\Con_1^N(X)$ under $H_i$ will have space of directions $X/H_i$. Since this does not have CBB(1), this cannot be an Alexandrov space \cite[Corollary 7.10]{BGP_SCBB_1_92}. 
Now, assuming that $X/H_i$ is not a spherical orbifold, then one can apply the same argument as before, and see that the space of directions of the image of the suspension point is not a spherical orbifold, and the space $\Con_1^N(X/H_i)$ cannot be an orbifold.
\end{remark}
In the following example, we use the same groups as in \cite{An_Yu_Yu_13}. However, we will arrange a situation where we can use Corollary \ref{Corollary :Cone statement} to provide an example of isospectral, non-isometric $\mathrm{RCD
}$ spaces which are not Alexandrov, which to the author’s knowledge, form the first such pair. 
\begin{exmp}\label{Example: isospectral RCD spaces}
    With the embeddings used in \cite[Theorem 1.5]{An_Yu_Yu_13}, consider the groups $G=\mathrm{SU}(6)$, $H_1=\mathrm{U}(3)$ and $H_2=\mathrm{Sp}(1) \times \mathrm{SO}(4)$ and $K=\{e\}$. Thanks to \cite[Theorem 1.5]{An_Yu_Yu_13}, we know $H_1$ and $H_2$ are $K$-equivalent, and \cite[Theorem 1.7]{An_Yu_Yu_13} $G/H_1$ and $G/H_2$ are isospectral, but are not homeomorphic. Give $\mathrm{SU}(6)$ the bi-invariant metric $g$, scale if needed to ensure that $\Ric(\mathrm{SU}(6)) \geq 34g $ and $\mathrm{diam}( \mathrm{SU}(6)) < \pi$. Thus applying Corollary \ref{Corollary :Cone statement} to $G$, we know that the quotients are isospectral $\mathrm{RCD}(35,36)$ spaces. We now show that $\mathrm{SU(6)}/H_i$ has a plane of curvature zero, thus the spaces $Y_i:=\mathrm{Con}_1^{35}(G)/H_i$ are not Alexandrov.

    Note that as $G/H_1$, $G/H_2$ are normal homogeneous, they have non-negative curvature; furthermore, the rank of $H_1$ and $H_2$ is $3$ and the rank of $SU(6)$ is $5$. By \cite{Berger} (see \cite[Lemma 1.2]{Wilking_Ziller_18} for a modern treatment), $G/H_1$ and $G/H_2$ do not admit positive curvature with their normal homogeneous metrics, and must have a plane of zero curvature. Therefore $Y_i$ is not Alexandrov, but is $\mathrm{RCD}$. Finally, we know that $Y_1$ and $Y_2$ are not isometric, since $G/H_1$ and $G/H_2$ are not homeomorphic \cite[Theorem 1.7]{An_Yu_Yu_13}.
\end{exmp}
We now produce a pair of isospectral Alexandrov spaces, again appealing to Corollary \ref{Corollary :Cone statement} and using Remark~\ref{Remark: Cone points remark} above. We will use the groups given as in Example \ref{Example: isospectral RCD spaces} \cite[Theorem 1.2]{Adelstein_Sandoval_17} \cite[Theorem 1.5]{An_Yu_Yu_13}, and the spherical join construction, to ensure that the space of directions is not a spherical orbifold, forcing the space to be an Alexandrov space that is not a Riemannian orbifold. We will rely on the following lemma, which is a consequence of \cite[Theorem 1.1]{Lytchak-Foertsch-08}. Recall, \textit{affine rank} of a space $X$ is defined as the supremum over all topological dimensions of affine spaces that admit an isometric embedding into $X$. Note that as affine rank is bounded above by topological dimension, then finite-dimensional Alexandrov spaces have finite affine rank.

\begin{lemma}\label{Lemma: Join lemma}
    Let $X,Y,Z$ be compact, finite-dimensional geodesic metric spaces with diameter at most $\pi$. If $X$ and $Y$ are non-isometric, then $X * Z$ and $Y*Z$ are non-isometric.
\end{lemma}
\begin{proof}
    We will prove the contrapositive. First, we will prove that if $X\times Z$ and $Y\times Z$ are isometric, then $X$ and $Y$ are isometric. We then prove the join statement. 

    Thanks to \cite[Theorem 1.1]{Lytchak-Foertsch-08}, we know that $X,Y,Z$ have a unique decomposition into a direct product, 
    \begin{align*}
        X &= \mathbb{R}^{x_0} \times X_1 \times ...\times X_r, \\
        Y &= \mathbb{R}^{y_0}\times Y_1 \times ...\times Y_s, \\
        Z &= \mathbb{R}^{z_0} \times Z_1 \times...
        \times Z_t.
    \end{align*}
    Thus, if $X\times Z \cong Y\times Z$, then 
    $$
\mathbb{R}^{x_0+z_0} \times X_1 \times ...\times X_r \times Z_1 \times ... \times Z_t \cong \mathbb{R}^{y_0+z_0} \times Y_1 \times ... \times Y_s \times Z_1 \times ... \times Z_t.
    $$
    We can see that since any $X_i,Y_j,Z_l$ is not isometric to the real line, then $x_0=y_0$. Furthermore, by uniqueness, where the irreducible factors are counted with multiplicity, we have
    $\{ X_1,...,X_r\} =\{Y_1,...,Y_s\}$, thus $X\cong Y.$

    Now assume that $X*Z\cong Y*Z$, then $C(X*Z)\cong C(Y*Z)$. We know that $C(X*Z) \cong C(X)\times C(Z)$, thus by the argument above, we see that $C(X) \cong C(Y)$, and therefore $X \cong Y$, proving the lemma.
\end{proof}
\begin{exmp}\label{Example: isospectral Alexandrov spaces}
    Take $Z:= (\mathbb{C}P^2,g_{FS})$, where $g_{FS}$ is the Fubini--Study metric, which has diameter $\frac{\pi}{2}$. Thus, $Z$ is a CBB$(1)$ space, but note that it is not a spherical orbifold, as it does not have constant curvature $1$.

    Now define $X:=\mathbb{S}^{11}*Z$, where $*$ denotes the spherical join. It is known that since both spaces are CBB$(1)$, then so is $X$ \cite{BGP_SCBB_1_92, Grove_Markvorsen_95}, and $\diam X \leq \pi$. Let $G=\mathrm{U}(6)$, $H_1=\mathrm{U}(3)$ and $H_2=\mathrm{Sp}(1)\times \mathrm{SO}(4)$. Then we let $H_i$ act on $X$ via the $\mathbb{S}^{11}$ factor and trivially on $Z$. We see that $X/H_i \cong \mathbb{S}^{11}/H_i \  * Z$, where $\mathbb{S}^{11}/H_i$ are the spaces studied in \cite[Theorem 1.2]{Adelstein_Sandoval_17}. Moreover the subgroups are $\{e\}$-equivalent \cite[Theorem 1.5]{An_Yu_Yu_13}. Now, applying Corollary \ref{Corollary :Cone statement} to $X$, we obtain that $Y_1:=\Con_1^{16}(X/H_1)$ and $Y_2:=\Con_1^{16}(X/H_2)$ are isospectral. 

    Furthermore, at each suspension point of $Y_i$, the space of directions is $\mathbb{S}^{11}/H_i \ *Z$, and $Z$ is not a spherical orbifold. Hence $Y_1$ and $Y_2$ are non-orbifolds. These spaces are Alexandrov: since  the join $X$ is CBB$(1)$, then both $X/H_i$ are CBB$(1)$, and therefore their cones $Y_i$ are CBB$(1)$ \cite{BGP_SCBB_1_92}. As the spaces are finite-dimensional Alexandrov spaces, Lemma \ref{Lemma: Join lemma} gives $\mathbb{S}^{11}/H_1*Z\not\cong \mathbb{S}^{11}/H_2*Z$. The spaces $Y_1$ and $Y_2$ are not isometric: if they were, as $C(Y_i)\cong C(X/H_i)\times\mathbb{R}$, using \cite{Lytchak-Foertsch-08}, we see that $C(X/H_1)\cong C(X/H_2)$. This contradicts the fact that $\mathbb{S}^{11}/H_1*Z\not\cong \mathbb{S}^{11}/H_2*Z$, see \cite[Theorem 1.2]{Adelstein_Sandoval_17} and Lemma \ref{Lemma: Join lemma}.
\end{exmp}

\begin{remark}
The spaces produced by Corollary~\ref{Corollary :Cone statement} are always simply connected (recall that we are assuming that our spaces are connected). Indeed, the spaces $\mathrm{Con}^{N}_{1}(X/H_{1})$ and $\mathrm{Con}^{N}_{1}(X/H_{2})$ of Corollary~\ref{Corollary :Cone statement} are simply connected whenever $X$ is connected. This applies to Examples~\ref{Example: isospectral RCD spaces} and \ref{Example: isospectral Alexandrov spaces}, and can be compared with \cite{Sutton_03}, where isospectral, non-isometric, simply connected \emph{homogeneous manifolds} were constructed.
\end{remark}

\printbibliography
\end{document}